\documentclass[12pt]{article}

\usepackage{verbatim}
\usepackage{setspace}
\usepackage{sectsty}
\subsectionfont{\normalsize}
\makeatletter
\renewcommand\section{\@startsection{section}{1}{\z@}%
                                  {-2.0ex \@plus -1ex \@minus -.2ex}%
                                  {2.0ex \@plus.2ex}%
                                  {\normalfont\normalsize\bfseries}}
\makeatother
\usepackage{fullpage}
\usepackage{url}
\usepackage{amsmath}
\usepackage{mathdots}
\usepackage[pdftex]{graphicx}
\newcommand{\ang}[1]{\langle#1\rangle}

\newcommand{\nat}{{\sf N}}

\newcommand{\xvec}[1]{\ifcase 3{#1} {\ang {x_1,x_2,x_3} } \else 
\ifcase 4{#1} {\ang{x_1,x_2,x_3,x_4}} \else {\ang {x_1,\ldots,x_{#1}}}\fi\fi}
\newcommand{\yvec}[1]{\ifcase 3{#1} {\ang {y_1,y_2,y_3} } \else 
\ifcase 4{#1} {\ang{y_1,y_2,y_3,y_4}} \else {\ang {y_1,\ldots,y_{#1}}}\fi\fi}
\newcommand{\zvec}[1]{\ifcase 3{#1} {\ang {z_1,z_2,z_3} } \else 
\ifcase 4{#1} {\ang{z_1,z_2,z_3,z_4}} \else {\ang {z_1,\ldots,z_{#1}}}\fi\fi}
\newcommand{\vecc}[2]{\ifcase 3{#2} {\ang { {#1}_1,{#1}_2,{#1}_3 } } \else
\ifcase 4{#1} {\ang { {#1}_1,{#1}_2,{#1}_3,{#1}_{4} } }
\else {\ang { {#1}_1,\ldots,{#1}_{#2}}}\fi\fi}
\newcommand{\veccd}[3]{\ifcase 3{#2} {\ang { {#1}_{{#3}1},{#1}_{{#3}2},{#1}_{{#3}3} } } \else
\ifcase 4{#1} {\ang { {#1}_{{#3}1},{#1}_{{#3}2},{#1}_{#3}3},{#1}_{{#3}4} }
\else {\ang { {#1}_{{#3}1},\ldots,{#1}_{{#3}{#2}}}}\fi\fi}
\newcommand{\veccz}[2]{\ifcase 3{#2} {\ang { {#1}_0,{#1}_2,{#1}_3 } } \else
\ifcase 4{#1} {\ang { {#1}_0,{#1}_2,{#1}_3,{#1}_{4} } }
\else {\ang { {#1}_0,\ldots,{#1}_{#2}}}\fi\fi}
\newcommand{\xve}[1]{\ifcase 3{#1} {x_1,x_2,x_3} \else 
\ifcase 4{#1} {x_1,x_2,x_3,x_4} \else {x_1,\ldots,x_{#1}}\fi\fi}
\newcommand{\yve}[1]{\ifcase 3{#1} {y_1,y_2,y_3} \else 
\ifcase 4{#1} {y_1,y_2,y_3,y_4} \else {y_1,\ldots,y_{#1}}\fi\fi}
\newcommand{\zve}[1]{\ifcase 3{#1} {z_1,z_2,z_3} \else 
\ifcase 4{#1} {z_1,z_2,z_3,z_4} \else {z_1,\ldots,z_{#1}}\fi\fi}
\newcommand{\ve}[2]{\ifcase 3#2 {{#1}_1,{#1}_2,{#1}_3} \else
\ifcase 4#2 {{#1}_1,{#1}_2,{#1}_3,{#1}_{4}}
\else {{#1}_1,\ldots,{#1}_{#2}}\fi\fi}
\newcommand{\ved}[3]{\ifcase 3#2 {{#1}_{{#3}1},{#1}_{{#3}2},{#1}_{{#3}3}} \else
\ifcase 4#2 {{#1}_{{#3}1},{#1}_{{#3}2},{#1}_{{#3}3},{#1}_{{#3}4}}
\else {{#1}_{{#3}1},\ldots,{#1}_{{#3}{#2}}}\fi\fi}
\newcommand{\fuve}[3]{
\ifcase 3#2
{{#3}({#1}_1),{#3}({#1}_2,{#3}({#1}_3)} \else
\ifcase 4#2
{{#3}({#1}_1),{#3}({#1}_2),{#3}({#1}_3),{#3}({#1}_4)}
\else
{{#3}({#1}_1),\ldots,{#3}({#1}_{#2})}\fi\fi}
\newcommand{\setmathchar}[1]{\ifmmode#1\else$#1$\fi}
\newcommand{\vlist}[2]{%
	\setmathchar{%
		\compound#2\one{#2}\two
		\ifcompound
			({#1}_1,\ldots,{#1}_{#2})
		\else
			\ifcat N#2
				({#1}_1,\ldots,{#1}_{#2})
			\else
				\ifcase#2
					({#1}_0)\or
					({#1}_1)\or
					({#1}_1,{#1}_2)\or 
					({#1}_1,{#1}_2,{#1}_3)\or
					({#1}_1,{#1}_2,{#1}_3,{#1}_4)\else 
					({#1}_1,\ldots,{#1}_{#2})
				\fi
			\fi
		\fi}}

\newif\ifcompound
\def\compound#1\one#2\two{%
	\def\one{#1}
	\def\two{#2}
	\if\one\two
		\compoundfalse
	\else
		\compoundtrue
	\fi}

\newcommand{\xwe}[1]{\ifcase 3{#1} {x_1\wedge x_2\wedge x_3} \else 
\ifcase 4{#1} {x_1\wedge x_2\wedge x_3\wedge x_4} \else {x_1\wedge \cdots \wedge
x_{#1}}\fi\fi}
\newcommand{\we}[2]{\ifcase 3#2 {\ang { {#1}_1\wedge {#1}_2\wedge {#1}_3 } } \else
\ifcase 4{#1} {\ang { {#1}_1\wedge {#1}_2\wedge {#1}_3\wedge {#1}_{4} } }
\else {\ang { {#1}_1\wedge \cdots\wedge {#1}_{#2}}}\fi\fi}

\newcommand{\st}{\mathrel{:}}

\newcommand{\floor}[1]{\left\lfloor{#1}\right\rfloor}

\newcommand{\s}[1]{\s_{#1}}

\newcommand{\monus}{\;\raise.5ex\hbox{{${\buildrel
    \ldotp\over{\hbox to 6pt{\hrulefill}}}$}}\;}

\newcommand{\infinity}{\infty}

\newcounter{savenumi}

\newtheorem{theoremfoo}{Theorem}[section] 
\newenvironment{theorem}{\pagebreak[1]\begin{theoremfoo}}{\end{theoremfoo}}

\newtheorem{lemmafoo}[theoremfoo]{Lemma}
\newenvironment{lemma}{\pagebreak[1]\begin{lemmafoo}}{\end{lemmafoo}}
\newtheorem{conjecturefoo}[theoremfoo]{Conjecture}

\newtheorem{conventionfoo}[theoremfoo]{Convention}
\newenvironment{convention}{\pagebreak[1]\begin{conventionfoo}\rm}{\end{conventionfoo}}

\newtheorem{porismfoo}[theoremfoo]{Porism}

\newtheorem{gamefoo}[theoremfoo]{Game}

\newtheorem{corollaryfoo}[theoremfoo]{Corollary}
\newenvironment{corollary}{\pagebreak[1]\begin{corollaryfoo}}{\end{corollaryfoo}}

\newtheorem{openfoo}[theoremfoo]{Open Problem}

\newtheorem{exercisefoo}{Exercise}

\newcommand{\fig}[1] 
{
 \begin{figure}
 \begin{center}
 \input{#1}
 \end{center}
 \end{figure}
}

\newtheorem{potanafoo}[theoremfoo]{Potential Analogue}

\newtheorem{notefoo}[theoremfoo]{Note}
\newenvironment{note}{\pagebreak[1]\begin{notefoo}\rm}{\end{notefoo}}

\newtheorem{notabenefoo}[theoremfoo]{Nota Bene}

\newtheorem{nttn}[theoremfoo]{Notation}
\newenvironment{notation}{\pagebreak[1]\begin{nttn}\rm}{\end{nttn}}

\newtheorem{empttn}[theoremfoo]{Empirical Note}

\newtheorem{examfoo}[theoremfoo]{Example}

\newtheorem{dfntn}[theoremfoo]{Def}
\newenvironment{definition}{\pagebreak[1]\begin{dfntn}\rm}{\end{dfntn}}

\newtheorem{propositionfoo}[theoremfoo]{Proposition}

\newenvironment{proof}
    {\pagebreak[1]{\narrower\noindent {\bf Proof:\quad\nopagebreak}}}{\QED}

\newcommand{\yyskip}{\penalty-50\vskip 5pt plus 3pt minus 2pt}
\newcommand{\blackslug}{\hbox{\hskip 1pt
        \vrule width 4pt height 8pt depth 1.5pt\hskip 1pt}}
\newcommand{\QED}{{\penalty10000\parindent 0pt\penalty10000
        \hskip 8 pt\nolinebreak\blackslug\hfill\lower 8.5pt\null}
        \par\yyskip\pagebreak[1]}

\newcommand{\BBB}{{\penalty10000\parindent 0pt\penalty10000
        \hskip 8 pt\nolinebreak\hbox{\ }\hfill\lower 8.5pt\null}
        \par\yyskip\pagebreak[1]}

\newtheorem{factfoo}[theoremfoo]{Fact}

\newenvironment{block}{\begin{list}{\hbox{}}{\leftmargin 1em
    \itemindent -1em \topsep 0pt \itemsep 0pt \partopsep 0pt}}{\end{list}}

\begin{document}

\newcommand{\KN}{K_{\nat}}
\newcommand{\NRE}{\hbox{NUM-RED-EDGES\ }}
\newcommand{\NBE}{\hbox{NUM-BLUE-EDGES\ }}
\newcommand{\RED}{\hbox{RED\ }}
\newcommand{\BLUE}{\hbox{BLUE\ }}
\newcommand{\REDns}{\hbox{RED}}
\newcommand{\BLUEns}{\hbox{BLUE}}

\title{How many ways can you make change: Some Easy Proofs}

\author{
\and
{William Gasarch}
\thanks{University of Maryland,
College Park, MD\ \ 20742,
\texttt{gasarch@cs.umd.edu}
}
\\ {\small Univ. of MD at College Park}
}

\date{}

\maketitle

\begin{abstract}
Given a dollar, how many ways are there to make change
using pennies, nickels, dimes, and quarters?
What if you are given a different amount of money?
What if you use different coin denominations?
This is a well known problem and formulas are known.
We present simpler proofs in several cases.
We use recurrences to derive formulas if the coin denominations
are $\{1,x,kx,rx\}$, and we use a simple proof using generating functions
to derive a formula for any coin set. 
\end{abstract}

\section{Introduction}

How many ways are there to make change of a dollar using 
pennies, nickels, dimes, and quarters?
This is a well known question in recreational math, serious math, and computer science.
We have observed three types answers in the literature and on the web.

\begin{enumerate}
\item
There are 242 ways to make change.
The author then points to a computer program or to the
actual list of ways to do it.
\item
The number of ways to make change for $n$ cents is the 
coefficient of $z^n$ in the power series for
$$\frac{1}{(1-z)(1-z^5)(1-x^{10})(1-z^{25})}$$
which can be worked out.
\item
This is known as the problem of finding the Coefficients of the Sylvester Denumerant
and is related to the Frobenius problem (given $n$ coins what is the largest
number that you cannot make change of). These papers tend to use advanced mathematics.
These papers give different approaches to obtain formulas for {\it any} coin set.
We will discuss these papers later; however, for now we just list some of them
\cite{agnden,alffrob,alonden,balden,beckfrob,bellden,komfrob,losden,serden}.
\end{enumerate}

\begin{definition}
If $S$ is a set of coin denominations then {\it the change function for $S$} is
the function that, on input $n$, outputs the number of ways to make change
for $n$ using the coins in $S$.
\end{definition}

In the first part of this paper we use recurrences to obtain simple derivations
for the change function
for two types of sets:
(1) $S=\{1,s,ks\}$ where $s,k\ge 2$, and 
(2) $S=\{1,s,ks,rs\}$ where $s\ge 2$, and $2\le k < r$.
As a corollary we obtain the case of pennies, nickels, dimes, and quarters.
In passing we solve the change-for-a-dollar problem by hand.

In the second part of this paper we use generating functions (in a simple way) to obtain  for {\it any} finite set $S$,
the change function. 

The formulas we derive are known; however, our proofs are simpler than those in the literature.

\section{General Definitions and Theorems}

\begin{convention}
Let $S$ be a non empty set of coins.
The number of ways to make 0 cents change is 1.
For all $n\le -1$ the number of ways to make $n$ cents change is 0.
\end{convention}

\begin{definition}\label{de:def}
Let $S=\{1<s<t<u\}$.
\begin{enumerate}
\item
$a_n$ is the number of ways to make change of $n$ cents using pennies.
Clearly $(\forall n)[a_n=1]$.
\item
$b_n$ is the number of ways to make change of $n$ cents using the first two coins (pennies and $s$-cent coins).
Clearly $(\forall n)[b_n = a_n + b_{n-s}]$. We use that $(\forall n\le -1)[a_n=0]$.
\item
$c_n$ is the number of ways to make change of $n$ cents using the first three coins (pennies, $s$-cent coins, and $t$-cent coins).
Clearly $(\forall n)[c_n = b_n + c_{n-t}]$. 
\item
$d_n$ is the number of ways to make change of $n$ cents using all four coins (pennies, $s$-cent coins, $t$-cent coins, and $u$-cent coins).
Clearly $(\forall n)[d_n = c_n + d_{n-u}]$. 
\end{enumerate}
\end{definition}

We do one example: Let $S=\{1,2,4,5\}$.  What is $d_{9}$?
\begin{enumerate}
\item
If one 5-cent coin is used then for the remaining four cents you must use either
one 4-cent coin; two 2-cents coins; one 2-cent coin and two pennies; or four pennies.
\item
If no 5-cent coins and one 4-cent coin is used then for the remaining five cents you must use either 
two 2-cents coins and one penny; one 2-cent coin and three pennies; or five pennies.
\item
If no 5-cent coins and two 4-cent coins are used then for the remaining one cent you must use one penny.
\item
If no 5-cent coins and no 4-cent coins are used then you must use either zero, one, two, three, or four
2-cent coins and the appropriate number of pennies. 
\end{enumerate}
Hence $d_{9}=4+3+1+5 = 12$.

Using the recurrence for $b_n$ 
and $(\forall n)[a_n=1]$, one can show the following.

\begin{theorem}\label{th:bn}
$(\forall n)[b_n = \floor{\frac{n}{s}} + 1].$
\end{theorem}

We can now solve the change-for-a-dollar problem by hand.
Let $S=\{1,5,10,25\}$.
We need to compute $d_{100}$.
We use the exact formula for $b_n$ and the recurrences
for $c_n$ and $d_n$.

$d_{100} = c_{100}+c_{75} + c_{50} + c_{25} + c_0$

$c_0=1$

$c_{25}= b_{25} + b_{15} + b_5 = 6 + 4 + 2 = 12$

$c_{50}= b_{50} + b_{40} + b_{30} + b_{20} + b_{10} + b_0 = 11 + 9 + 7 + 5 + 3 + 1 = 36$

$c_{75}= b_{75} + b_{65} + b_{55} + b_{45} + b_{35} + c_{25} = 16 + 14 + 12 + 10 + 8 + 12 = 72$

$c_{100} = b_{100} + b_{90} + b_{80} + b_{70} + b_{60} + c_{50} = 21 + 19 + 17 + 15 + 13 + 36 =  121$

Hence

$d_{100} = 1 + 12 + 36 + 72 + 121 = 242$.

\section{The Coin Set $\{1,s,ks\}$}

Throughout this section we will be using the coin set
$S=\{1,s,ks\}$ where $s,k\ge 2$ are fixed natural numbers.
The quantities $a_n,b_n,c_n$ are as in Definition~\ref{de:def}
with coin set $S$.

To determine the number of ways to make change, you can always round
down to the nearest multiple of $s$. Formally $c_{sL+L_0}=c_{sL}$.
We use this without mention.

Let $n=sL+L_0$ where $0\le L_0\le s-1$ and $L\ge 1$.
Using the recurrence for $c_n$ and the formula for $b_n$ (from Theorem~\ref{th:bn}) we have:

\[
\begin{array}{rl}
c_n= c_{sL} =  & b_{sL} + c_{s(L-k)} \cr
                       =  & b_{sL} + b_{s(L-k)}+ c_{s(L-2k)} \cr
		       =  & b_{s(L-0)} + b_{s(L-k)}+\cdots + b_{s(L-ki)}+c_{s(L-ki-k)}\cr
		       =  & (L+1) + (L-k+1) + \cdots + (L-ki+1) + c_{s(L-ki-k)}\cr
\end{array}
\]

Let $L\equiv j \pmod k$. Let $i=\frac{(L-j-k)}{k}$.
Then the last term in the sum is $c_{sj}$.
Since $j\le k-1$, $sj < sk$. Hence
$c_{sj} = b_{sj} = j+1$. 
The resulting sum is an arithmetic series with
first term $j+1$, 
last term  $j+1 + (\frac{L-j}{k})k$,
and  number of terms $\frac{L-j}{k}+1 = \frac{L-j+k}{k}$.
Hence after easy algebra we have the following

\begin{theorem}\label{th:cn}
Let $n=sL+L_0$ where $0\le L_0\le s-1$ and $L\ge 1$.
(So that $L=\floor{\frac{n}{s}}$.)
\begin{enumerate}
\item
Let $j$ be such that $L\equiv j \mod k$. Then 
$$c_n=\frac{L^2+(k+2)L + 2k}{2k} + \frac{(k-2)j-j^2}{2k}.$$
\item
$\frac{n^2}{2ks^2} + \frac{n}{2s}-k \le c_n \le  \frac{n^2}{2ks^2} + \frac{(k+2)n}{2ks} + \frac{(k-2)^2}{8}+1$
\end{enumerate}
\end{theorem}

\begin{proof}
Part 1 follows from our work.
We prove Part 2

For the lower bound we use that $L\ge \frac{n-s}{s}$ and note that the last term has min value, as $0\le j\le k-1$,
of $-k$.
For the upper bound we use that $L\le \frac{n}{s}$ and note that the last term is max value, as $0\le j\le k-1$, of
$\frac{(k-2)^2}{8}+1$.
\end{proof}

\begin{note}
Theorem~\ref{th:cn} for the special case of pennies, nickels, and dimes
($s=5$, $k=2$) was proven by Deborah Levine's article~\cite{makingchange}.
\end{note}

\begin{note}
One can derive $c_n=\frac{n^2}{2ks^2} + \Theta(\frac{n}{2k})$
from Schur's theorem~\cite{schurpost,schurgen,wilfgen}.
\end{note}

\section{The Coin Set $\{1,s,ks,rs\}$}

Throughout this section we will be using the coin set
$S=\{1,s,ks,rs\}$ where $s,k,r$ are fixed natural numbers with
$s,k,r\ge 2$ and $r>k$.
The quantities $a_n,b_n,c_n,d_n$ are as in Definition~\ref{de:def}
with coin set $S$.

To determine the number of ways to make change, you can always round
down to the nearest multiple of $s$. Formally $d_{sL+L_0}=d_{sL}$.
We use this without mention.

Let $n=s(rL+M)+L_0$ where $0\le M\le r-1$, $0\le L_0\le s-1$.
Using the recurrence for $d_n$ we have:

\[
\begin{array}{rl}
d_n= d_{s(rL+M)}   =  & c_{s(rL+M)} + d_{s(rL+M-r\times 1)}\cr
                   =  & c_{s(rL+M)} + c_{s(rL+M-r\times 1)} + d_{s(rL+M-r\times 2)}\cr
		   =  & c_{s(rL+M)} + c_{s(rL+M-r\times 1)} + c_{s(rL+M-r\times 2)} + \cdots + c_{s(M+r)}+ d_{sM}\cr
		   =  & c_{s(rL+M)} + c_{s(rL+M-r\times 1)} + c_{s(rL+M-r\times 2)} + \cdots + c_{s(M+r)}+ c_{sM}\cr
		   =  & \sum_{i=0}^{L} c_{s(ri+M)}\cr
\end{array}
\]

Using the formula for $c_n$ from Theorem~\ref{th:cn} we obtain

$$
d_n=\sum_{i=0}^L \frac{(M+ri)^2+(k+2)(M+ri)+2k}{2k}+\sum_{i=0}^L \frac{(k-2)(ri+M \bmod k)-(ri+M\bmod k)^2}{2k}.
$$

We will evaluate the second sum later. For now we name it:

\begin{notation}
$\Delta(L,M)= \sum_{i=0}^L \frac{(k-2)(ri+M \bmod k)-(ri+M\bmod k)^2}{2k}$
\end{notation}

Thus $d_n$ is

$$
\sum_{i=0}^L \frac{(M+ri)^2+(k+2)(M+ri)+2k}{2k}+\Delta(L,M).
$$

$$
=\frac{1}{2k}
\biggl ((L+1)(M^2+kM+2M+2k) + r(2M+k+2)\sum_{k=0}^L i + r^2\sum_{i=0}^L i^2\biggr )  + \Delta(L,M)
$$

$$
=\frac{1}{12k}
\biggl ((L+1)(
2r^2L^2+
(r^2+6Mr+3kr+6r)L+
6M^2+
(6k+12)M+
12k
)\biggr )  + \Delta(L,M)
$$

\begin{lemma}
Let $L,M\ge 1$ and $a\ge 0$. 
\begin{enumerate}
\item
$\sum_{i=0}^L (ri+M \bmod k)^a= \sum_{j=0}^{k-1} (rj+M)^a \floor{\frac{L-j+k}{k}}.$
\item
$\Delta(L,M) = \frac{1}{2k}(\sum_{j=0}^{k-1} (\floor{\frac{L-j}{k}}+1)((k-2)(rj+M \bmod k) - (rj+M \bmod k)^2)$
\item
If $k=2$ and $r\equiv 0 \pmod 2$ then $\Delta(L,M)=-\frac{(1+(-1)^{M+1})(L+1)}{8}$.
\item
If $k=2$ and $r\equiv 1 \pmod 2$ then $\Delta(L,M)= -\frac{2L+ (1+(-1)^L)(1+(-1)^{M+1})+(1+(-1)^{L+1})}{16}$.
\end{enumerate}
\end{lemma}

\begin{proof}

\noindent
1) We break this sum into parts depending on what $i$ is congruent to mod $k$.

\[
\begin{array}{rl}
\sum_{i=0}^L (ri+M \bmod k)^a=& \sum_{j=0}^{k-1} \sum_{i=0,i\equiv j \bmod k}^L (ri+M \bmod k)^a\cr
                           =& \sum_{j=0}^{k-1} \sum_{i=0,i\equiv j \bmod k}^L (rj+M \bmod k)^a\cr
                           =& \sum_{j=0}^{k-1} (rj+M \bmod k)^a\sum_{i=0,i\equiv j \bmod k}^L 1\cr
                           =& \sum_{j=0}^{k-1} (rj+M \bmod k)^a \floor{\frac{L-j+k}{k}} \cr
\end{array}
\]

\bigskip

\noindent
2) This follows from part 1 using $a=1$ and $a=2$.

\bigskip

\noindent
3 and 4) If $k=2$ then notice that the expression for $\Delta(L,M)$ is simplified considerably
since $k-2=0$ and $(rj+M\bmod 2)^2 = (rj+M\bmod 2)$. Also note that the summation only has two terms
($j=0$ and $j=1$). Hence we obtain

$$\Delta(L,M) = -\frac{1}{4}
\biggl ( \floor{\frac{L+2}{2}} (M \bmod 2) + \floor{\frac{L+1}{2}}(M\bmod 2)\biggr ).$$

\noindent
{\bf Case 0: $r\equiv 0 \pmod 2$}. 

If $M\equiv 0 \pmod 2$ then $\Delta(L,M)=0$.

If $M\equiv 1 \pmod 2$ then 

$$\Delta(L,M) = -\frac{1}{4}
\biggl ( \floor{\frac{L+2}{2}} + \floor{\frac{L+1}{2}}\biggr )=-\frac{L+1}{4}
$$

One can check that $\Delta(L,M)=-\frac{(1+(-1)^{M+1})(L+1)}{8}$.

\noindent
{\bf Case 1: $r\equiv 1 \pmod 2$} Then

$$\Delta(L,M) = -\frac{1}{4}
\biggl ( \floor{\frac{L+2}{2}} (M \bmod 2) + \floor{\frac{L+1}{2}}(1+M\bmod 2)\biggr ).$$

The following table summarizes what $\Delta(L,M)$ is, given what $L,M$ are mod 2.

\[
\begin{array}{|c|c|c|}
\hline
L \bmod 2 &M \bmod 2 &  \Delta(L,M)\cr
\hline
 0 & 0 & -\frac{L}{8}\cr
 0 & 1 & -\frac{L+2}{8} \cr
 1 & 0 & -\frac{L+1}{8} \cr
 1 & 1 & -\frac{L+1}{8} \cr
\hline
\end{array}
\]

One can check that $\Delta(L,M)=\frac{2L+ (1+(-1)^L)(1+(-1)^{M+1})+(1+(-1)^{L+1})}{16}$.
\end{proof}

Putting this all together we have the following.

\begin{theorem}\label{th:dn}
Let $n=s(rL+M)+L_0$ where $0\le M\le r-1$, $0\le L_0 \le s-1$, and $L\ge 1$.
(So $L=\floor{\frac{n}{rs}}$, $M=\floor{\frac{n\bmod {rs}}{s}}$,
and $n\equiv L_0 \pmod s$.)
Then the following are true.
\begin{enumerate}
\item
$d_n$ is
$$
\frac{1}{12k}
\biggl ((L+1)(
2r^2L^2+
(r^2+6Mr+3kr+6r)L+
6M^2+
(6k+12)M+
12k
)\biggr )
$$

$$+\frac{1}{2k}\biggl (\sum_{j=0}^{k-1} \biggl (\floor{\frac{L-j}{k}}+1\biggr )((k-2)(rj+M \bmod k) - (rj+M \bmod k)^2)\biggr )$$

\item
$d_n= \frac{n^3}{6krs^3} + \Theta(n^2)$.

\item
If $k=2$ and $r\equiv 0 \pmod 2$ then $d_n$ is
$$
\frac{1}{24}
\biggl ((L+1)(
2r^2L^2+
(r^2+6Mr+12r)L+
6M^2+
24M+
24
)\biggr )
-\frac{(1+(-1)^{M+1})(L+1)}{8}
$$

\item
If $k=2$ and $r\equiv 1 \pmod 2$ then $d_n$ is
$$
\frac{1}{24}
\biggl ((L+1)(
2r^2L^2+
(r^2+6Mr+12r)L+
6M^2+
24M+
24
)\biggr )
$$

$$
+\frac{2L+ (1+(-1)^L)(1+(-1)^{M+1})+(1+(-1)^{L+1})}{16}.
$$
\end{enumerate}
\end{theorem}

\begin{note}
One can derive the asymptotic result (part 2) from Schur's theorem~\cite{schurpost,schurgen,wilfgen}.
\end{note}

As a corollary of Theorem~\ref{th:dn} we obtain a formula
for making change of $n$ cents
using pennies, nickels, dimes, and quarters.

\begin{corollary}
If $s=5$, $k=2$, and $r=5$ then $d_n$ is 
$$
\frac{1}{24}
\biggl ((L+1)(
50L^2+
(8530M)L+
6M^2+
24M+
24
)\biggr )
$$

$$
+\frac{2L+ (1+(-1)^L)(1+(-1)^{M+1})+(1+(-1)^{L+1})}{16}.
$$
\end{corollary}

\section{Any Finite Coin Set}

Throughout this section $S=\{t_1<t_2<\cdots < t_v\}$ 
is our coin set. We will derive the change function for $S$ using
generating functions. The use of generating functions is well known in this area.
The basic idea of this proof is from Graham, Knuth, Patashnik~\cite{concrete}.
The proof we give seems to be new. We discuss the literature after we obtain the result.


It is easy to see that the number of ways to make change of $n$ cents using the coins in $S$ is
the coefficient of $z^n$ in 

$$C(z)=\prod_{i=1}^v \frac{1}{(1-z^{t_i})}$$

Let $t$ be the least common multiple of $\{t_1,\ldots,t_v\}$.
For $1\le i\le v$ let $f_i$ be the polynomial such that 
$(1-z^{t})= (1-z^{t_i})f_i(z)$. Note for later that

$$f_i(z) = (1+z^{t_i} + z^{2t_i} + \cdots + z^{t}).$$

It is easy to see that

$$C(z)=\frac{f_1(z)\cdots f_v(z)}{(1-z^{t})^{v}}.$$

Let $A(z)=f_1(z) \cdots f_v(z)$. 
Let $M$ be the degree of $A(z)$ which is 
$((t-t_1)+(t-t_2)+\cdots+(t-t_v))\le tv$.
Let the coefficient of $z^{j}$ in $A(z)$ be $a_j$.
Using the power series expansion 
$\frac{1}{(1-x)^{v}} = \sum_{i=0}^\infinity \binom{i+v-1}{v-1} x^i$ and
plugging in $x=z^t$ we obtain

$$
C(z)=
\biggl (\sum_{j=0}^{M} a_j z^{j} \biggr )
\biggl ( \sum_{i=0}^\infinity \binom{i+v-1}{v-1} z^{ti} \biggr )
=
\sum_{j=0}^M \sum_{i=0}^\infinity a_j \binom{i+v-1}{v-1} z^{ti+j}.
$$

The coefficient of $z^n$  is 

$$\sum_{0\le j\le M \st j\equiv n \bmod t}  a_j \binom{\frac{n-j}{t}+v-1}{v-1}$$

How easy is it to find the $a_j$'s? Note that since 

$$f_i(z) = (1+z^{t_i} + z^{2t_i} + \cdots + z^{t})$$

\noindent
$a_j$ is the number of ways to make change of $j$ using coins in $S$ with the
restriction that coin $t_i$ is used at most $\frac{t}{t_i}$ times.
A simple dynamic program can find all of the $a_j$'s; however, this will
take $t^v$ steps.

How many operations does it take to, given $n$, find the answer.
Note that
the binomial coefficients are consecutive in that the top part goes through
$\le \frac{M}{t}$ consecutive numbers while the bottom part stays the same.
Hence this can be done in $\le v-1 + \frac{M}{t}\le 2v$ multiplications.
each one is multiplied by the appropriate $a_j$ which is another $\le \frac{M}{t}\le v$ multiplications.
Hence we have $\le 3v$ multiplications.
The summation then adds $\frac{M}{t}\le v$ additions
(one would need to code this up carefully and only visit those $j\equiv n \bmod t$).
This leads to $4v$ operations.
The number of operations depends only on the number of coins and not their
values; however, if the coins had large values that would make each operation take longer.

Putting this all together we have the following.
\begin{theorem}\label{th:fml}
Let $S=\{t_1,\ldots,t_v\}$ and $t$ be the least common multiple of $t_1,\ldots,t_v$.
Let $M=((t-t_1)+(t-t_2)+\cdots+(t-t_v))$.
Then the change function for $S$ is
$$\sum_{0\le j\le M \st j\equiv n \bmod t}  a_j \binom{\frac{n-j}{t}+v-1}{v-1}$$
where the coefficients $a_j$ can be found in $O(t^v)$ steps and
the formula itself can be evaluated in $O(v)$ steps.
\end{theorem}

\begin{definition}
A function $f(n)$ is {\it quasi polynomial} if there exists a polynomial $g$, a function $h$,
and a number $B$ such that $f(n)=g(n)+h(n)$ and $h(n)$ depends only on $n\bmod B$.
The polynomial $g$ is called {\it the polynomial part}.
\end{definition}

We obtain the following (known) corollary of Theorem~\ref{th:fml}.
\begin{corollary}\label{co:lcm}
For any set $S$ the change function of $S$ is quasi polynomial with $B$ being
the least common multiple of the elements of $S$.
\end{corollary}

Corollary~\ref{co:lcm} was know in the days of Sylvester and Cayley 
(see the references in~\cite{bellden}).
E.T. Bell~\cite{bellden} 
provided a different proof which is fairly simple but uses roots of unity. 
We believe our proof of Corollary~\ref{co:lcm} is simpler. 
Komatsu~\cite{komfrob} gives a method to determine,
given coin set $S$ and the number $n$
how many ways are there to make change. He claims his method is computationally
practical.  I presume it is faster than our method; however, he is not explicit about how
long it takes. The calculations involve roots of unity and partial derivatives.
We believe our approach is simpler. 
Losonek~\cite{losden} claims to have an exact formula for the case of $v=3$;
however, since the paper is not online, whatever he has will be lost to future generations.
Beck, Gessel, and Komatsu~\cite{beckfrob} derive a general formula for the polynomial
part of the change function. Their polynomial depends on Bernoulli numbers.
They do not obtain a general formula; hence our result and theirs are incomparable.
Baldoni, Berline, De Loera, Dutra, and Vergne~\cite{balden} have a polynomial time
algorithm for the following: for fixed $k$, compute the first highest $k+1$ coefficients
of the change function (they define this carefully). Their algorithm uses rather sophisticated mathematics.
Our approach is simpler but our algorithm to obtain the change function is slower.
Tripathi~\cite{tripathifrob} gives a simple proof of a general formula.
His formula depends on parameters
$m_j$ that are the least $N$ such that $N\equiv  \bmod {t_1}$ and
one can make change for $N$. 
Our proof and his are different. While we believe ours is simpler, this
is debatable.

Can one obtain a formula for the change problem quickly?
Since the unbounded knapsack problem (you are allowed to use any item any number of times)
is NP-complete, it is unlikely that a formula can be obtained and evaluated quickly.

\section{Acknowledgment}

I thank 
Tucker Bane,
Adam Busis,
and
Clyde Kruskal
for proofreading and suggestions.
I thank Michele Vergne for pointing me to the rich literature of
the change problem.
I thank Eric Weaver who wrote a program
that tested the formulas and hence gave me confidence in them.


\end{document}